\documentclass[review, onefignum, onetabnum]{siamltex}
\usepackage{graphicx, color, rotating, varwidth, multirow}
\usepackage{array, booktabs}
\usepackage{lipsum}
\usepackage{amsfonts}
\usepackage{graphicx}
\usepackage{epstopdf}
\usepackage{algorithmic}
\usepackage{enumitem}
\usepackage{amsopn}
\usepackage{amsmath}
\usepackage[ruled,vlined]{algorithm2e}
\usepackage{hyperref}
\usepackage{cleveref}
\allowdisplaybreaks[4]
\usepackage{geometry}
\newgeometry{left=2cm,right=2cm,top=2cm,bottom=2cm}

\setlist[enumerate]{leftmargin=.5in}
\setlist[itemize]{leftmargin=.5in}

\newcommand{\tr}[1]{\mathrm{tr}\left(#1\right)}
\renewcommand*{\H}{\mathrm{H}}

\newcommand{\fnorm}[1]{\left\| #1 \right\|_{\mathrm{F}}}
\newcommand{\norm}[1]{\left\| #1 \right\|}

\newcommand{\bbC}{\mathbb{C}}
\newcommand{\bbE}{\mathbb{E}}
\newcommand{\bbR}{\mathbb{R}}

\newcommand{\calE}{\mathcal{E}}

\newcommand{\tSigma}{\widetilde{\Sigma}}
\newcommand{\tLambda}{\widetilde{\Lambda}}
\newcommand{\tD}{\widetilde{D}}
\newcommand{\tG}{\widetilde{G}}
\newcommand{\tP}{\widetilde{P}}
\newcommand{\tQ}{\widetilde{Q}}
\newcommand{\tR}{\widetilde{R}}
\newcommand{\tvartheta}{\tilde{\vartheta}}
\newcommand{\tvarphi}{\tilde{\varphi}}
\newcommand{\tchi}{\tilde{\chi}}

\title{Fast randomized algorithms for low-rank matrix approximations with
applications in global comparative analysis of a class of data
sets~\thanks{The research of this work was partially supported by the
National Key R\&D Program of China 2020YFA0711902, Natural Science
Foundation of China under Grant No.~11971243, 12071159, 11690013,
U1811464. This work was also supported in part by the major key project of
Peng Cheng Laboratory under grant PCL2023AS1-2.}}

\author{
Weiwei Xu~\thanks{School of Mathematics and Statistics, Nanjing
University of Information Science and Technology, Nanjing 210044,
China. The Peng Cheng Laboratory, Shenzhen 518055, China, and with the
Pazhou Laboratory (Huangpu), Guangzhou 510555, China.} \and
Weijie Shen~\thanks{School of Mathematics and Statistics, Nanjing
University of Information Science and Technology, Nanjing 210044,
China.} \and
Wen Li~\thanks{School of Mathematical Sciences, South China Normal
University, Guangzhou 510631, China.} \and
Weiguo Gao~\thanks{School of Mathematical Sciences and School of Data
Science, Fudan University, Shanghai 200433, China.} \and
Yingzhou Li~\thanks{School of Mathematical Sciences, Fudan University,
Shanghai 200433, China.}
}

\begin{document}

\maketitle

\begin{abstract}
Generalized singular values (GSVs) play an essential role in the
comparative analysis. In the real world data for comparative analysis,
both data matrices are usually numerically low-rank. This paper
proposes a randomized algorithm to first approximately extract bases
and then calculate GSVs efficiently. The accuracy of both basis
extration and comparative analysis quantities, angular distances,
generalized fractions of the eigenexpression, and generalized
normalized Shannon entropy, are rigursly analyzed. The proposed
algorithm is applied to both synthetic data sets and the genome-scale
expression data sets. Comparing to other GSVs algorithms, the proposed
algorithm achieves the fastest runtime while preserving sufficient
accuracy in comparative analysis.
\end{abstract}



\section{Introduction}

The generalized singular value decomposition (GSVD) is a valuable and
versatile mathematical tool in various scientific fields, including but
not limited to, the general Gauss-Markov linear model, real-time signal
processing, image processing. In the past decades, GSVD has played an
essential role in the comparative analysis of genome-scale expression
data, DNA-sequence, and mRNA-expression data~\cite{abb00, abb03, d98,
el89, z92}. Motivated by the comparative analysis, we propose a randomized method to accelerate the GSVD computation therein, which could be
applied to other applications as well.

In the comparative analysis of expression data sets, we are given two
expression data sets, $G_1$ and $G_2$. The data sets $G_1$ and $G_2$ are
of size $m$-genes $\times$ $n$-arrays and $p$-genes $\times$ $n$-arrays
respectively. In order to distinguish the similarities and dissimilarities between two expression data sets, the GSVD is introduced as a mathematical
tool. GSVD simultaneously transforms $G_1$ and $G_2$ to two reduced
$n$-genelets $\times$ $n$-arraylets spaces~\cite{abb00}. Then, various
similarity measurements are adopted to compare the two data sets.

Throughout this paper, we adapt the following definitions of Grassman
matrix pair (GMP) and GSVD~\cite{gv13, ps81}. For $G_1 \in \bbC^{m \times
n}$ and $G_2 \in \bbC^{p \times n}$, the matrix pair $\{G_1, G_2\}$ is an
$(m, p, n)$ \emph{Grassman matrix pair} if $\rank{(G_1^\H, G_2^\H)} = n$.
Now we consider an $(m,p,n)$-GMP $\{G_1, G_2\}$. Its GSVD is defined
as~\footnote{This GSVD is actually reduced GSVD under the tall rectangular
version, which is the one used in comparative analysis. Our proposed
method is able to address non-reduced GSVD efficiently as well.},
\begin{equation} \label{eq:gsvd1}
    G_1 = U \Sigma_{G_1} R, \quad
    G_2 = V \Sigma_{G_2} R,
\end{equation}
where $U \in \bbC^{m \times n}$ and $V \in \bbC^{p \times n}$ are column orthogonal matrices, $R \in \bbC^{n \times n}$ is a nonsingular matrix, $\Sigma_{G_1}
\in \bbR^{n \times n}$ and $\Sigma_{G_2} \in \bbR^{n \times n}$ are
diagonal matrices with generalized singular values on their diagonals,
i.e., $\Sigma_{G_1} = \diag{\alpha_1, \ldots, \alpha_n}$ and
$\Sigma_{G_2} = \diag{\beta_1, \ldots, \beta_n}$ with
\begin{equation}\label{ab}
    \begin{split}
        1 = \alpha_1 = \cdots = \alpha_r
        > \alpha_{r+1} \geq \cdots \geq \alpha_{r+s}
        > \alpha_{r+s+1} = \cdots = \alpha_n = 0,\\
        0 = \beta_1 = \cdots = \beta_r
        < \beta_{r+1} \leq \cdots \leq \beta_{r+s}
        < \beta_{r+s+1} = \cdots = \beta_n = 1,
    \end{split}
\end{equation}
and $\alpha_i^2 + \beta_i^2 = 1$ for $1 \leq i \leq n$. Here $r$ and
$n-r-s$ are numbers of zeros in $\{\beta_i\}$ and $\{\alpha_i\}$
respectively, and $s$ counts the number that both $\alpha_i$ and $\beta_i$
are not zeros. In the rest of the paper, we refer to $s$ as the rank of
the GSVD.

Once the GSVD of the expression data sets $G_1$ and $G_2$ is obtained, the
matrix $R$ defines the $n$-arraylets $\times n$-arrays basis
transformation that is shared by both data sets. Matrices $U$ and $V$
define the $m$-genes $\times$ $n$-genelets and $p$-genes $\times$
$n$-genelets basis transformation for $G_1$ and $G_2$ respectively. With
these basis transformations, the original comparison between $G_1$ and
$G_2$ would be carried out by the comparison of $\{\alpha_i\}_{i=1}^n$ and
$\{\beta_i\}_{i=1}^n$. The relative significance of the $\ell$-th genelet,
i.e., the significance of the $\ell$-th genelet in $G_1$ compared to that
in $G_2$, is determined by the ratio of $\alpha_\ell$ and $\beta_\ell$.
We denote the relative significance as
\begin{equation*}
    \rho_\ell = \frac{\alpha_\ell}{\beta_\ell}
\end{equation*}
for $\ell = r+1, \dots, n$, and $\rho_\ell = \infty$ for $\ell = 1, \dots,
r$. Besides the relative significance, there are other measurements used
in the comparative analysis: antisymmetric angular distance, generalized
fractions of eigenexpression, and generalized normalized Shannon entropy.

The \emph{antisymmetric angular distance} for the $\ell$-th genelet
between $G_1$ and $G_2$ is defined as,
\begin{equation} \label{eq:antisymmetric-angular-distance}
    \vartheta_\ell = \arctan \left( \frac{\alpha_\ell}{\beta_\ell} \right)
    - \frac{\pi}{4}.
\end{equation}
An antisymmetric angular distance of $\vartheta_\ell = 0$ indicates that
the $\ell$-th genelet is of equal significance in both data sets. While,
the distance $\vartheta_\ell = \pi/4$ indicates that the $\ell$-th genelet
in $G_1$ is significant relative to $G_2$, whereas $\vartheta_\ell =
-\pi/4$ indicates the other way around, i.e., the $\ell$-th genelet in
$G_2$ is significant relative to $G_1$. The antisymmetric angular
distances are ordered as $\pi/4 \geq \vartheta_1 \geq \cdots \geq
\vartheta_n \geq -\pi/4$.

The \emph{generalized fractions of eigenexpression} of
$G_1$ and $G_2$ are defined as,
\begin{flalign} \label{eq:generalized-fraction}
    P_{1,\ell} = \alpha_\ell^2/\sum_{k=1}^n \alpha_k^2,
    \quad
    P_{2,\ell} = \beta_\ell^2/\sum_{k=1}^n \beta_k^2,
\end{flalign}
respectively for $\ell = 1, \dots, n$. The generalized fractions of
eigenexpression is not a relative distance between $G_1$ and $G_2$. The
fraction $P_{i,\ell}$ indicates the significance of the $\ell$-th genelet
in $G_i$ for $i=1,2$. Note that the generalized fractions of
eigenexpression $P_{1,\ell}$ and $P_{2,\ell}$ can be viewed as the
probability that a genelet in $G_1$ and $G_2$ respectively.

The \emph{generalized normalized Shannon entropy},
\begin{flalign} \label{eq:generalized-shannon-entropy}
    D_i = \frac{-1}{\log{n}} \sum_{k=1}^n
    P_{i,k}\log{P_{i,k}},
\end{flalign}
for $i = 1, 2$, defines an entropy
measurement for the generalized fractions for $G_1$ and $G_2$. By the property of entropy, we have $D_i \in [0, 1]$. The
 generalized normalized Shannon entropy measures the complexity of
 expression of genelets in the data set. If $D_i = 0$, then all expressions
 are captured by a single genelet in $G_i$. If $D_i = 1$, then expressions
 are in a disordered status, and all genelets in $G_i$ are equally
 expressed.

Numerical methods of GSVD have been well developed. The GSVD of two real
matrices was first proposed by Van Loan~\cite{23}. Paige and
Saunders~\cite{ps81} used the CS decomposition of the unitary matrix to
propose GSVD of matrix pair, which extended the real matrices in \cite{23}
to complex matrices. Bai and Demmel~\cite{5} described a variation of
Paige's algorithm for computing the GSVD with an extra preprocessing step
and a new algorithm in addressing $2 \times 2$ triangular GSVD.
Stewart~\cite{21} and Van Loan~\cite{24} proposed two backward stable
algorithms for computing the GSVD. Ewerbring and Luk~\cite{el89} and
Zha~\cite{z92} extended GSVD for matrix triplets. Recently,
Friedland~\cite{f05} proposed a new GSVD algorithm, which suppresses the
sensitivity to an error in the entries of the matrices. Xu et.
al.~\cite{xnb20} proposed the geometric inexact Newton method for
generalized singular values of the Grassmann matrix pair. The GSVD of the
matrix pair in MATLAB is calculated using the CS decomposition described
in \cite{gv13} and the built-in SVD and QR functions.

In this paper, we first propose a low-rank approximation algorithm based
on random sampling technique with QR decomposition with pivoting. The
randomized low-rank algorithm is then applied to approximately extract the
column bases of $G_1$ and $G_2$ matrices. On top of the basis extraction,
we propose \cref{alg:rand-gsv} to obtain GSVs. The approximation accuracy
of the basis extraction is analyzed in \cref{thm:projector} and the
accuracy mainly depends on the decay property of the GSVs. Combined with
the perturbation analysis of GSVs, we derive the accuracy analysis for
quantities in comparative analysis. Finally, on both synthetic data sets
and practical genome-scale expression data sets, the proposed algorithm
shows advantages in runtime. And the accuracy is way beyond the desired
ones in comparative analysis tasks.

The rest of the paper is organized as follows. In
\cref{sec:randomized-gsvd}, a randomized method is proposed to compute the
GSVs of $(m, p, n)$-GMPs. Then, the
generalized fractions of eigenexpression and generalized normalized
Shannon entropy for comparative analysis of two data sets are calculated
and analyzed in \cref{sec:gsvd-analysis}. In \cref{sec:numerical-results},
numerical results for both synthetic data sets and practical yeast and human cell-cycle expression data sets are reported to demonstrate the efficiency of the proposed randomized method. Finally,
\cref{sec:conclusion} concludes the paper with some discussions on future
work.

\section{Randomized algorithms for low-rank matrix approximations for GSVs}
\label{sec:randomized-gsvd}

In this section, Gaussian random matrices are used to construct randomized
algorithms with low-rank matrix approximations to remove the near-zero
GSVs, either $\alpha$ or $\beta$, and reduce the overall computational
cost. In the following, we first give a detailed description of our
randomized algorithms for GSVs. Then, the computational cost comparison is
discussed.

The randomized algorithm for GSVs is composed of two phases: 1) randomized
algorithm for basis extraction; 2) calculating GSVs for compressed matrix
pair.

The randomized algorithm for basis extraction aims to find an orthonormal
basis sets for $U$ and $V$ in \cref{eq:gsvd1} with non-zero GSVs
$\alpha_i$ and $\beta_j$, respectively. Our randomized algorithm is
essentially the same as the basis extraction algorithm in randomized SVD \cite{p21}.
We need to apply the randomized algorithm to $G_1$ and $G_2$, and obtain
an approximated basis of $U$ and $V$ with non-zero GSVs. The difference
mainly lies in the later analysis in \cref{sec:gsvd-analysis}. Our goal of
the randomized algorithm is approximating $U$ and $V$ whereas the original
randomized SVD aims to approximate the left or right singular vectors of
the matrix. Hence, as we will see later, the condition number of $R$ would
get into play in our approximation error analysis. In order to be
self-contained, we will describe the randomized algorithm in extracting
the basis.

Since the sizes of the approximated basis of $U$ and $V$ are unknown in a
priori, we conduct an iterative scheme to obtain the basis batch by batch.
We could also calculate the basis one by one, which is less efficient in
modern computer architecture. Hence, we define a blocksize hyperparameter
$b$ controlling the batch size to benefit from the memory hierarchy
efficiency. In many environments, picking $b$ between 10 and 100 would be
near optimal. \cite{p21} In our numerical experiments, we set $b$ to be 100.
For each iteration in the basis extraction, we apply the
matrix to a Gaussian random matrix of size $m \times b$, followed by a
projection matrix projecting out the bases from previous iterations. Then,
we apply the reduced QR factorization to the matrix product result and
obtain another batch of bases. We seek to build an orthonormal matrix $Q$
such that
\begin{equation*}
    \fnorm{ (I - QQ^\H) G } < \epsilon,
\end{equation*}
which is adopted as the stopping criterion. The square of the Frobinius
norm of $G - QQ^\H G$ could be calculated in a cumulated way efficiently.
Hence, the dominant computational cost of the basis extraction algorithm
lies in applying the matrix $G$ to Gaussian random matrices. We summarize
the basis extraction algorithm in \cref{alg:basis-ext}.

\begin{algorithm}
    \caption{Randomized algorithms for basis extraction}
    \begin{algorithmic}[1]\label{alg:basis-ext}
        \REQUIRE Given an $m\times n$ matrix $G$, a tolerance $\epsilon$
        and a blocksize integer number $b$.

        \ENSURE An approximated basis $Q$ of $U$ for $G = U \Sigma R$ as in
        \cref{eq:gsvd1}.

        \STATE $Q = [~].$

        \FOR{$i = 1$ to $n/b$}

            \STATE Let $\Omega_i$ be a Gaussian random matrix of size $n
            \times b$.

            \STATE Evaluate the projected matrix $Y_i = (I - Q Q^\H) (G
            \Omega_i).$ \label{alg:basis-ext:mat-mat}

            \STATE Compute the reduced QR decomposition $Y_i = P_i T_i$
            for $P_i \in \bbC^{m \times b}$, $T_i \in \bbC^{b \times b}$.

            \STATE Append $P_i$ to $Q$, i.e., $Q = [Q, P_i]$.

            \IF{$\fnorm{(I - Q Q^\H)G} < \epsilon$}

                \STATE Return $Q$.

            \ENDIF

        \ENDFOR

    \end{algorithmic}
\end{algorithm}

Then we aim to calculate GSVs for the compressed matrix pair and obtain
the GSVD of the original matrix pair as in \cref{eq:gsvd1}. We now
consider a scenario that generalized singular values are explicitly
divided into three groups: exactly one, between one and zero, and exactly
zero. Then the GSVD under the tall rectangular version admits,
\begin{equation} \label{eq:gsvd3}
    G_1 =
    \begin{pmatrix}
        U_1 & U_2 & U_3
    \end{pmatrix}
    \begin{pmatrix}
        I & & \\
        & \tSigma_1 & \\
        & & 0 \\
    \end{pmatrix}
    R, \quad
    G_2 =
    \begin{pmatrix}
        V_1 & V_2 & V_3
    \end{pmatrix}
    \begin{pmatrix}
        0 & & \\
        & \tSigma_2 & \\
        & & I \\
    \end{pmatrix}
    R.
\end{equation}
The left bases of $G_1$, $G_2$ are $Q_1$, $Q_2$. In the \cref{eq:gsvd3},
$Q_1$ is the basis of $\begin{pmatrix} U_1 & U_2 \end{pmatrix}$ and
orthogonal to $U_3$, and $Q_2$ is the basis of $\begin{pmatrix}V_2 & V_3
\end{pmatrix}$ and orthogonal to $V_1$. Taking the projection of $G_1$ and
$G_2$ on the basis of $Q_1$ and $Q_2$ respectively, we obtain,
\begin{equation} \label{eq:gsvd-reduced}
    \begin{pmatrix}
        G_1 \\ G_2
    \end{pmatrix}
    =
    \begin{pmatrix}
        Q_1 Q_1^\H G_1 \\
        Q_2 Q_2^\H G_2
    \end{pmatrix}
    =
    \begin{pmatrix}
        Q_1 & \\
        & Q_2
    \end{pmatrix}
    \begin{pmatrix}
        Q_1^\H G_1 \\
        Q_2^\H G_2
    \end{pmatrix}
    =
    \begin{pmatrix}
        Q_1 & \\
        & Q_2
    \end{pmatrix}
    \begin{pmatrix}
        P_1
        \begin{pmatrix}
            I & & 0 \\
            & \tSigma_1 & 0
        \end{pmatrix} R \\
        P_2
        \begin{pmatrix}
            0 & \tSigma_2 & \\
            0 & & I
        \end{pmatrix} R \\
    \end{pmatrix},
\end{equation}
where $P_1 = Q_1^\H \begin{pmatrix} U_1 & U_2 \end{pmatrix}$ and $P_2 =
Q_2^\H \begin{pmatrix} V_2 & V_3 \end{pmatrix}$ are square unitary
matrices. The last equality in \eqref{eq:gsvd-reduced} obeys a general
form of GSVD for $\begin{pmatrix} Q_1^\H G_1 \\ Q_2^\H G_2 \end{pmatrix}$.
Hence the following phase calculates the GSVs for the compressed matrix
pair $Q_1^\H G_1$ and $Q_2^\H G_2$. In the second algorithm, we apply
reduced QR factorization to obtain the column basis of the matrix pair,
and then calculate the GSVs from the basis directly. More precisely, let
$L_1$ and $L_2$ be the top and bottom parts of the partial unitary matrix
of the reduced QR factorization,
i.e.,
\begin{equation*}
    \begin{pmatrix}
    Q_1^\H G_1 \\
    Q_2^\H G_2
    \end{pmatrix}
    =
    \begin{pmatrix}
    L_1 \\
    L_2
    \end{pmatrix}
    \tR.
\end{equation*}
where $L_1 \in \bbC^{l_1 \times n}$, $L_2 \in \bbC^{l_2 \times n}$ forms a
partial unitary matrix, and $\tR \in \bbC^{n \times n}$ is an upper
triangular matrix. The singular values of either $L_1$ or $L_2$ would
reveal the GSVs of our original problem. Hence, we compute the singular
values of the one of $L_1$ and $L_2$ with a smaller matrix size, and then
calculate the GSV pairs. The overall algorithm is summarized in
\cref{alg:rand-gsv}, where the basis extraction
algorithm~(\cref{alg:basis-ext}) is denoted as ``BasisExt'', the tolerance
and blocksize are passed to the function implicitly.

\begin{algorithm}
\caption{Randomized GSVs Algorithm}
\begin{algorithmic}[1] \label{alg:rand-gsv}
    \REQUIRE Given matrix pair $G_1 \in \bbC^{m \times n}$, $G_2 \in
    \bbC^{p \times n}$.

    \ENSURE Generalized singular values $\{\alpha_i\}_{i=1}^n$,
    $\{\beta_i\}_{i=1}^n$.

    \STATE $Q_1 = \mathrm{BasisExt}(G_1)$.

    \STATE $Q_2 = \mathrm{BasisExt}(G_2)$.

    \STATE Compute the reduced QR decomposition
    $
    \begin{pmatrix}
    Q_1^\H G_1 \\
    Q_2^\H G_2
    \end{pmatrix}
    =
    \begin{pmatrix}
    L_1 \\
    L_2
    \end{pmatrix}
    \tR$.

    \STATE Denote numbers of rows of $L_1$ and $L_2$ as $l_1$ and $l_2$
    respectively.

    \IF{$l_1 \leq l_2$}

        \STATE Compute the singular values $L_1$, denoted as
        $\{\alpha_i\}_{i=1}^{l_1}$.

        \STATE Append zeros $\alpha_i = 0$ for $i = l_1+1, \dots,
        n$.

        \STATE Calculate $\beta_i = \sqrt{1 - \alpha_i^2}$ for $i = 1,
        \dots, n$.

    \ELSE

        \STATE Compute the singular values $L_2$ and sort them
        ascendingly, denoted as $\{\beta_i\}_{i=n-l_2+1}^n$.

        \STATE Append zeros $\beta_i = 0$ for $i = 1, \dots, n - l_2$.

        \STATE Calculate $\alpha_i = \sqrt{1 - \beta_i^2}$ for $i = 1,
        \dots, n$.

    \ENDIF

\end{algorithmic}
\end{algorithm}

If we want to recover the GSVD, we can do the following: Let the diagonal
matrices composed of generalized singular values calculated by
\cref{alg:rand-gsv} be $\Sigma_{G_1}$ and $\Sigma_{G_2}$. Let the singular
value decomposition of $L_1$ and $L_2$ be $L_1=U_1\Sigma_{G_1}W_1$ and
$L_2=U_2\Sigma_{G_2}W_2$ respectively. If $l_1\leq l_2$, the matrices of
GSVD of matrix pair $\{G_1,G_2\}$ are $U=Q_1U_1$,
$V=Q_2L_2W_1^{-1}\Sigma_{G_2}^{-1}$ and $R = W_1 \tR$ as in
\cref{eq:gsvd1}. If $l_1> l_2$, the matrices of GSVD of the matrix pair
$\{G_1,G_2\}$ are $U=Q_1L_1W_2^{-1}\Sigma_{G_1}^{-1}$, $V=Q_2V_2$ and $R = W_2
\tR$ as in \cref{eq:gsvd1}.

We now analyze the computational complexities of \cref{alg:basis-ext} and
\cref{alg:rand-gsv}. In \cref{alg:basis-ext}, the most expensive steps are
the matrix-matrix multiplication (line~\ref{alg:basis-ext:mat-mat}). The
computational cost could be estimated as
\begin{equation*}
    \sum_{i = 1}^{l/b} O(m n b) = O(mnl),
\end{equation*}
where $l$ is the number of columns in the output $Q$. The computational
cost of \cref{alg:rand-gsv} could be divided into three parts: basis
extraction, reduced QR factorization, and SVD calculation. The basis
extraction cost is the cost of \cref{alg:basis-ext} applying to $G_1$ and
$G_2$, and admits $O(mnl_1) + O(pnl_2)$. The cost for SVD calculation is
$O(n\min(l_1, l_2)^2)$. The cost for the reduced QR factorization step
composed of two matrix-matrix multiplications and a QR factorization,
\begin{equation*}
    O(mnl_1) + O(pnl_2) + O((l_1+l_2)n^2),
\end{equation*}
which dominates the cost of the other two parts and is the overall cost
for \cref{alg:rand-gsv}.
In contrast, without basis compression, the cost of calculating the GSVs
of the matrix pair $G_1$ and $G_2$ would be dominated by the QR
factorization as that in \cref{alg:rand-gsv}, and admits,
\begin{equation*}
    O((m + p)n^2).
\end{equation*}
Consider a tall rectangular version of GSVD, the costs of GSVs
calculations, with and without basis compression, differ by a ratio of
$\max\{l_1, l_2\}/n$ in the complexity analysis. Further, the leading cost
of \cref{alg:rand-gsv} comes from the matrix-matrix multiplication,
whereas that for GSVs without basis compression comes from the QR
factorization. The extra prefactor difference between matrix-matrix
multiplication and QR factorization is the extra saving for our proposed
algorithm.

\section{Comparative analysis of a class of genome-scale expression data sets}
\label{sec:gsvd-analysis}

For a given matrix $M$, we write $P_{M}$ for the unique orthogonal
projector with $\mathrm{range}(P_{M})=\mathrm{range}(M)$. When $M$ has
full column rank, we can express this projector explicitly
\begin{eqnarray*}\label{p0}
P_{M} = M(M^\H M)^{-1}M^\H.
\end{eqnarray*}

In \cref{alg:basis-ext}, for a matrix $A$, $A(:,i:j)$ denotes the submatrix from the $i$-th column to the $j$-th column in $A$, and $A(k,k)$
denotes the $k$-th diagonal element of $A$.

\begin{lemma}\label{c2} Let $Q_1$ and $\epsilon$ be given by
\cref{alg:basis-ext}, then in step 7 there exist $i$ such that
$\fnorm{Q_1 Q_1^{\H} G_1 - G_1}<\epsilon$, and there exist $j$ such that $\fnorm{Q_2 Q_2^{\H}
G_2 - G_2} < \epsilon$.
\end{lemma}
\begin{proof}
Here we only introduce the proof of $G_1$ in detail.
Assume that
$G_1\Omega$ has the reduced QR decomposition
\begin{equation*}
    G_1 \Omega = G_1 (\Omega_1, \cdots, \Omega_{\frac{n}{b}})
    = \tQ R = (\tQ_1, \cdots, \tQ_{\frac{n}{b}}) R
    = (\tQ_1, \cdots, \tQ_{\frac{n}{b}})
    \begin{pmatrix}
        R_{11} & \cdots & R_{1,\frac{n}{b}} \\
        & \ddots & \vdots \\
        &  & R_{\frac{n}{b},\frac{n}{b}} \\
    \end{pmatrix},
\end{equation*}
where $\Omega_i$ is the $n\times b$ submatrix of $\Omega$. When $i=1$,
$Y_1=G_1\Omega_1$, $Y_1$ has the reduced QR decomposition
$Y_1=P_1 T_1$. Due to $Y_1$ being a column full rank matrix, the QR
decomposition of $Y_1$ is unique. Observe that $P_1 = \tQ_1$,
$T_1=R_{11}$. When $i=2$, $Y_2 = G_1 \Omega_2 - \tQ_1
\tQ_1^{\H} G_1 \Omega_2$ has the reduced QR decomposition $Y_2 =
P_2 T_2$. For $G_1 (\Omega_1, \Omega_2) = (\tQ_1, \tQ_2)
\begin{pmatrix}
    R_{11} & R_{12} \\
    0 & R_{22} \\
\end{pmatrix}$.
It follows that
\begin{equation*}
    G_1 \Omega_1 = \tQ_1 R_{11}, G_1 \Omega_2 = \tQ_1 R_{12} + \tQ_2 R_{22}
\end{equation*}
and
\begin{equation*}
\begin{aligned}
\begin{pmatrix}
    R_{11} & R_{12} \\
    0 & R_{22} \\
\end{pmatrix}
&=
\begin{pmatrix}
    \tQ_1^\H \\
    \tQ_2^\H \\
\end{pmatrix}
\begin{pmatrix}
    \tQ_1 & \tQ_2 \\
\end{pmatrix}
\begin{pmatrix}
    R_{11} & R_{12} \\
    0 & R_{22} \\
\end{pmatrix}
=
\begin{pmatrix}
    \tQ_1^\H \\
    \tQ_2^\H \\
\end{pmatrix}
\begin{pmatrix}
    G_1\Omega_1 & G_1\Omega_2 \\
\end{pmatrix} \\
& =
\begin{pmatrix}
    \tQ_1^\H G_1\Omega_1 & \tQ_1^\H G_1 \Omega_2 \\
    \tQ_2^\H G_1\Omega_1 & \tQ_2^\H G_1 \Omega_2 \\
\end{pmatrix}.
\end{aligned}
\end{equation*}
Observe that $\tQ_2 R_{22} = G_1 \Omega_2 - \tQ_1 R_{12} = G_1 \Omega_2 -
\tQ_1 \tQ_1^{\H} G_1 \Omega_2$. Since $Y_2$ is a column full rank matrix,
the QR decomposition of $Y_2$ is unique. So $P_2 = \tQ_2$, $T_2 = R_{22}$.
The same is true when $i > 2$. Therefore, $P_i = \tQ_i, 1 \leq i \leq
\frac{n}{b}$.

By \cref{alg:basis-ext} if $i = \frac{n}{b}$, then $Q_1 = [P_1, \ldots,
P_{\frac{n}{b}}] = [\tQ_1, \cdots, \tQ_{\frac{n}{b}}]$. Since
$\Omega$ is an $n\times n$ standard Gaussian matrix with
$\rank{\Omega}=n$, then by $G_1 \Omega = (\tQ_1, \cdots,
\tQ_{\frac{n}{b}}) R = Q_1 R$ we have $G_1 = Q_1 R \Omega^{-1}$.
Hence, $Q_1 Q_1^{\H} G_1 = Q_1 Q_1^{\H} Q_1 R \Omega^{-1} = Q_1 R
\Omega^{-1} = G_1$. Then there exist $Q_1$ such that for precision
$\epsilon$ we have $\fnorm{Q_1 Q_1^{\H} G_1 - G_1}<\epsilon$.
\end{proof}


Next, we will analyze the accuracy of the basis extraction.

\begin{proposition}[Proposition 10.1 \cite{hmt11}] \label{prop:sgt}
    Fix matrices $S$, $T$, and draw a standard Gaussian matrix $G$. Then
    \[\bbE\fnorm{SGT}^2 = \fnorm{S}\fnorm{T}.\]
\end{proposition}

\begin{proposition}[Proposition 10.2 \cite{hmt11}] \label{prop:t1}
    Draw a $k \times (k+p)$ standard Gaussian matrix $G$ with $k\geq2$ and
    $p\geq2$. Then
    \[\bbE \fnorm{G^{\dag}}^2 = \frac{k}{p-1}.\]
\end{proposition}

\begin{theorem}[Theorem 3.3.16 \cite{hj91}] \label{thm:singular}
    Let $A, B \in \mathbb{C}^{m\times n}$ be given. The following
    inequalities hold for the decreasingly ordered singular values of $A$,
    $B$ and $AB^{\H}$.
    \begin{equation*}
        \sigma_i(AB^{\H}) \leq \sigma_i(A) \sigma_1(B), \quad i = 1, 2,
        \ldots, \min\{m, n\}.
    \end{equation*}
\end{theorem}

\begin{theorem} \label{thm:projector}
    Let $G_1 \in \bbC^{m \times n}$ and $G_2 \in \bbC^{p \times n}$ satisfy
    \cref{eq:gsvd1}, the target ranks $k_1, k_2 \geq 2$ and the
    oversampling parameters $p_1, p_2 \geq 2$ obey $k_1 + p_1 \leq
    \min\{m, n\}$, $k_2 + p_2 \leq \min\{p, n\}$, and $\Omega_1 \in
    \bbC^{n \times (k_1 + p_1)}$ and $\Omega_2 \in \bbC^{n \times (k_2 +
    p_2)}$ be standard Gaussian matrices. Denote $\varphi_i$ and $\chi_i$
    as GSVs of $\{G_1, G_2\}$. For $Q_1 \in \bbC^{m \times (k_1 + p_1)}$
    and $Q_2 \in \bbC^{m \times (k_2 + p_2)}$ calculated by
    \cref{alg:basis-ext}, we have
    \begin{equation*}
        \bbE \fnorm{\left(I_m - Q_1 Q_1^{\H} \right) G_1}^2
        \leq \eta \left(\frac{k_1}{p_1 - 1} + 1\right)
        \sum_{j > k_1}^n \varphi^2_j, \quad \text{and}
    \end{equation*}
    \begin{equation*}
        \bbE \fnorm{\left(I_p - Q_2 Q_2^{\H} \right) G_2}^2
        \leq \eta \left(\frac{k_2}{p_2 - 1} + 1\right)
        \sum_{j > k_2}^n \chi^2_{n-j+1},
    \end{equation*}
    where $\eta = \sigma_{\max}(G_1^{\H} G_1 + G_2^{\H} G_2)$.
\end{theorem}

\begin{proof}
Throughout this proof, we focus on the analysis of $G_1$ and $Q_1$ and
omit the subscript for simplicity. We inherit the GSVD of $\{G_1, G_2\}$
as in \cref{eq:gsvd1}. Let the SVD of $G$ be $G = \hat{U} \hat{\Sigma}
\hat{V}^{\H}$. Then, we have,
\begin{equation*}
    G (G_1^{\H} G_1 + G_2^{\H} G_2)^{-\frac{1}{2}} = \hat{U} \hat{\Sigma}
    \hat{V}^{\H} (G_1^{\H} G_1 + G_2^{\H} G_2)^{-\frac{1}{2}}
    = U \Sigma_{G} W,
\end{equation*}
where $W = R (R^{\H} R)^{-\frac{1}{2}}$ is a unitary matrix. Rewriting
$\hat{\Sigma}$ in terms of $\Sigma_G$, we obtain,
\begin{equation*}
    \hat{\Sigma} = \hat{U}^{\H} U \Sigma_{G} W (G_1^{\H} G_1 + G_2^{\H}
    G_2)^{\frac{1}{2}} \hat{V}.
\end{equation*}
The SVD of $G$ could be rewritten as top and bottom parts,
\begin{equation}
    G = \hat{U} \hat{\Sigma} \hat{V}^{\H} = \hat{U}
    \begin{pmatrix}
        \hat{\Sigma}_t & \\
        & \hat{\Sigma}_b \\
    \end{pmatrix}
    \begin{pmatrix}
        \hat{V}_t^{\H} \\
        \hat{V}_b^{\H} \\
    \end{pmatrix},
\end{equation}
where $\hat{\Sigma}_t \in \bbR^{k \times k}$ is a diagonal matrix with the
largest $k$ GSVs of $G$ on the diagonal, $\hat{\Sigma}_b \in \bbR^{(m-k)
\times (n-k)}$ is a diagonal matrix with the rest GSVs on the diagonal,
$\hat{V}_t \in \bbC^{n \times k}$ and $\hat{V}_b \in \bbC^{n \times
(n-k)}$ form a compatible top-bottom partition of $\hat{V}$.

By the unitary invariant property of the Frobinius norm, we have,
\begin{equation*}
\begin{split}
    \fnorm{ \left(I - QQ^{\H} \right) G}
    &= \fnorm{ \left(I - \hat{U}^\H P_{Q} \hat{U} \right) \hat{U}^\H G}\\
    &= \fnorm{ \left(I - \hat{U}^\H P_{G\Omega} \hat{U} \right) \hat{U}^\H G}
    = \fnorm{ \left(I - P_{\hat{U}^\H G\Omega} \right) \hat{U}^\H G},
\end{split}
\end{equation*}
where $P_{\hat{U}^\H G \Omega}$ denotes the projector formed by $\hat{U}^\H
G\Omega$. We further construct an approximated basis of $\hat{U}^\H
G\Omega$ as,
\begin{equation*}
    Z = \hat{U}^\H G \Omega \Lambda^\dagger \hat{\Sigma}_t^{-1} =
    \begin{pmatrix}
        I \\ F
    \end{pmatrix},
\end{equation*}
where
\begin{equation*}
    \Lambda = \hat{V}_t^{\H} \Omega, \quad
    F = \hat{\Sigma}_{b} \tLambda \Lambda^\dagger
    \hat{\Sigma}_t^{-1}, \quad \text{and }
    \tLambda = \hat{V}_b^{\H} \Omega.
\end{equation*}

From the expression of $Z$, we have $\mathrm{range}(Z) \subset
\mathrm{range}(\hat{U}^\H G\Omega)$ and, hence, obtain
\begin{equation} \label{eq:projGbound}
    \fnorm{\left(I - Q Q^\H \right) G}
    = \fnorm{ \left(I - P_{\hat{U}^\H G\Omega} \right) \hat{U}^\H G}
    \leq \fnorm{ \left(I - P_Z \right) \hat{U}^\H G}.
\end{equation}
The projector $I - P_Z$ could be explicitly written in terms of $F$,
\begin{equation*}
    I - P_{Z} =
    \begin{pmatrix}
        I - (I + F^{\H}F)^{-1} & B\\
        B^\H & I - F (I + F^{\H} F)^{-1} F^{\H}
    \end{pmatrix},
\end{equation*}
where $B = -(I + F^{\H} F)^{-1} F^{\H}$. Since
\begin{equation*}
    I - (I + F^{\H}F)^{-1} \preceq F^{\H} F, \quad \text{and }
    I - F(I + F^{\H}F)^{-1} F^{\H} \preceq I,
\end{equation*}
we could give an upper bound for the projector $I - P_Z$,
\begin{equation*}
    I - P_Z \preceq
    \begin{pmatrix}
    F^{\H}F & B\\
    B^{\H} & I
    \end{pmatrix}.
\end{equation*}
Then substituting the upper bound of $I - P_Z$ into \cref{eq:projGbound},
we have
\begin{equation*}
    \begin{split}
        \fnorm{ \left(I - QQ^{\H}\right)G}^2 &
        \leq \fnorm{ (I - P_{Z}) \hat{\Sigma} \hat{V}^{\H}}^2
        = \fnorm{(I - P_{Z}) \hat{\Sigma}}^2\\
        & \leq  \tr{
        \begin{pmatrix}
            \hat{\Sigma}_{t} &  \\
            & \hat{\Sigma}_{b} \\
        \end{pmatrix}
        \begin{pmatrix}
            F^{\H}F & B\\
            B^{\H} & I
        \end{pmatrix}
          \begin{pmatrix}
            \hat{\Sigma}_{t} &  \\
            & \hat{\Sigma}_{b} \\
        \end{pmatrix}}
        =
        \fnorm{F \hat{\Sigma}_{t}}^2 + \fnorm{\hat{\Sigma}_{b}}^2.
    \end{split}
\end{equation*}
Taking the expectation with
respect to the randomness in $\Omega$, we prove the inequality for $G =
G_1$,
\begin{equation*}
    \begin{split}
        \bbE \fnorm{ \left(I - QQ^{\H}\right)G}^2
        \leq  \bbE \fnorm{\hat{\Sigma}_{b} \tLambda \Lambda^{\dag}}^2
        + \fnorm{\hat{\Sigma}_{b}}^2,
    \end{split}
\end{equation*}
where the definition of $F$ is substituted. By the definition of $\Lambda$
and $\tLambda$, they are the top and bottom parts of the unitary matrix
$\hat{V}$ applied to the standard Gaussian matrix $\Omega$. Due to the
property of the standard Gaussian matrix, we know that $\Lambda$ and
$\tLambda$ are independent. Hence, we compute this expectation by first
conditioning on $\Lambda$ and then computing the expectation with respect
to $\Lambda$,
\begin{equation*}
    \bbE \fnorm{\hat{\Sigma}_b \tLambda \Lambda^{\dag}}^2
    = \bbE \left(\bbE \left[\fnorm{\hat{\Sigma}_b \tLambda
    \Lambda^{\dag}}^2 \middle| \Lambda \right]\right)
    = \bbE \left( \fnorm{\hat{\Sigma}_b}^2
    \fnorm{\Lambda^{\dag}}^2 \right)
    = \frac{k}{p-1} \cdot \fnorm{\hat{\Sigma}_b}^2,
\end{equation*}
where the second equality is due to \cref{prop:sgt} and the last equality
is due to \cref{prop:t1}.

Combined with singular value inequality~\cref{thm:singular}, we prove the
first expectation inequality in the theorem,
\begin{equation*}
    \bbE\fnorm{(I-QQ^{\H})G}^2 \leq \left(1+\frac{k}{p-1}\right)
    \fnorm{\hat{\Sigma}_b}^2
    \leq \eta\left(1+\frac{k}{p-1}\right) \sum_{j>k}\varphi_{j}^2,
\end{equation*}
where $\eta = \sigma_{\max}(R^{\H}R) = \sigma_{\max}(G_1^{\H} G_1 +
G_2^{\H} G_2)$. The second expectation inequality in the theorem could be
proved similarly.
\end{proof}

In the following, we estimate the numerical errors in comparative analysis
quantities when the generalized singular values are perturbed. We stick to
the following notations,
\begin{eqnarray*}
    G =
    \begin{pmatrix}
        G_1 \\
        G_2 \\
    \end{pmatrix},
    \quad
    \tG = G + \Delta G = G +
    \begin{pmatrix}
        \Delta G_1 \\
        \Delta G_2\\
    \end{pmatrix}
    =
    \begin{pmatrix}
        \tG_1 \\
        \tG_2 \\
    \end{pmatrix}.
\end{eqnarray*}
The generalized singular value pairs $(\varphi_i, \chi_i)$ of $\{G_1,
G_2\}$ and those $(\tvarphi_i, \tchi_i)$ of $\{\tG_1, \tG_2\}$ be ordered
as in \cref{ab}. The errors between between $\varphi_\nu, \chi_\nu$
and $\tvarphi_\nu, \tchi_\nu$ are denoted as $\Delta \varphi_\nu$,
$\Delta \chi_\nu$ respectively.

From \cref{thm:projector}, we know that the randomized GSVs algorithm
could produce fairly accurate $\tG_1$ and $\tG_2$ for a small tolerance
$\epsilon$ and relatively large $k_1$ and $k_2$, hence, small $\Delta G$.
We introduce a new notation $\calE$ as,
\begin{equation*}
    \calE = \sqrt{2} \fnorm{\Delta G}
    \min \left\{ \norm{G^{\dagger}}, \norm{\tG^{\dagger}} \right\},
\end{equation*}
which will be used to bound the numerical errors for both GSVs and
comparative analysis quantities.

\begin{lemma} \cite{p84} \label{lem:per}
    Assume $\rank{G} = \rank{\tG} = n$. Then
    \begin{equation*}
        \sqrt{\sum_{i=1}^n \left[(\varphi_i - \tvarphi_i)^2 + (\chi_i -
        \tchi_i)^2\right]} \leq \calE.
    \end{equation*}
\end{lemma}

Through a direct calculation, we could have the following error bounds
based on \cref{lem:per}.

\begin{corollary} \label{cor:error-bounds}
    Assume $\rank{G} = \rank{\tG} = n$. Then
    \begin{equation*}
        |\Delta \varphi| \leq \calE \quad \text{and} \quad
        |\Delta \chi| \leq \calE,
    \end{equation*}
    where $|\Delta \varphi| = \max_{1 \leq \nu \leq n}\{|\Delta
    \varphi_\nu|\}$ and $|\Delta \chi| = \max_{1 \leq \nu \leq n}\{|\Delta
    \chi_\nu|\}$.
\end{corollary}

\begin{theorem} \label{thm:exp}
    Let $\vartheta_\nu$, $P_{1,\nu}$, $P_{2,\nu}$, $D_1$ and $D_2$
    represent the exact values, and $\tvartheta_\nu$, $\tP_{1,\nu}$,
    $\tP_{2,\nu}$, $\tD_1$ and $\tD_2$ represent the values calculated by
    \cref{alg:rand-gsv}, \cref{eq:antisymmetric-angular-distance},
    \cref{eq:generalized-fraction} and
    \cref{eq:generalized-shannon-entropy}. Then for $\nu = 1, 2, \ldots, n$,
    \begin{flalign*}
        (i)&\;
        \left|\vartheta_\nu - \tvartheta_\nu \right|
        \leq \arcsin(2 \calE), \\
        (ii)&\;
        \left|P_{1,\nu} - \tP_{1,\nu}\right| \leq
        \frac{2 \varphi_\nu \calE} {\sum_{k=1}^n \varphi_k^2}
        + o\left( \calE\right) \quad \text{and} \quad
        \left|P_{2,\nu} - \tP_{2,\nu}\right| \leq
        \frac{2\chi_\nu \calE}{\sum_{k=1}^n \chi_{k}^2}
        + o\left( \calE\right), \\
        (iii)&\;
        \left|D_1 - \tD_1\right| \leq 2\calE \sum_{i=1}^n
        \left|\frac{\varphi_i}{\sum_{k=1}^n \varphi_k^2}
        \left(\log\frac{\varphi_i^2}{\sum_{k=1}^n
        \varphi_k^2}\frac{1}{\log{n}}+D_1\right) \right| + o(\calE) \quad
        \text{and}\\
        &\;
        \left|D_2 - \tD_2\right| \leq 2\calE \sum_{i=1}^n
        \left|\frac{\chi_i}{\sum_{k=1}^n \chi_k^2}
        \left(\log\frac{\chi_i^2}{\sum_{k=1}^n\chi_k^2}\frac{1}{\log{n}} +
        D_2\right) \right| + o(\calE).
    \end{flalign*}
\end{theorem}

\begin{proof}
Let $\varphi_\nu = \sin(\gamma_\nu)$, $\chi_\nu = \cos(\gamma_\nu)$,
$\tvarphi_\nu = \varphi_\nu + \Delta \varphi_\nu = \sin(\gamma_\nu +
\Delta\gamma_\nu)$ and $\tchi_\nu = \chi_\nu + \Delta\chi_\nu =
\cos(\gamma_\nu + \Delta\gamma_\nu)$, where $|\Delta\varphi_\nu| \leq 1$,
$|\Delta\chi_\nu| \leq 1$, $\Delta\gamma_\nu$ is a perturbation. Without
loss of generality, we denote $\Theta=\diag{\theta_1, \ldots, \theta_n}$
with $\theta_\nu = \varphi_\nu$ as the case for $l_1 \leq l_2$ in
\cref{alg:rand-gsv}. Recall the expressions of $\vartheta_\nu$,
$P_{i,\nu}$ and $D_i$ in terms of $\theta_\nu$,
\begin{eqnarray*}
    \vartheta_\nu & = & \arctan \left(\frac{\theta_\nu}
    {\sqrt{1 - \theta_\nu^2}}\right) - \frac{\pi}{4}, \quad
    P_{1,\nu} = \frac{\theta_\nu^2}{\tr{\Theta^2}}, \quad
    P_{2,\nu} = \frac{1 - \theta_\nu^2}{n - \tr{\Theta^2}}, \\
    && D_1 = -\frac{1}{\log{n}} \left[ \sum_{i=1}^n
    \frac{\theta_i^2}{\tr{\Theta^2}}
    \left[2 \log{\theta_i} - \log\left(\tr{\Theta^2}\right) \right]\right],\\
    && D_2 = -\frac{1}{\log{n}} \left[ \sum_{i=1}^n
    \frac{1 - \theta_i^2}{n - \tr{\Theta^2}}
    \left[ \log\left(1-\theta_i^2\right) -
    \log\left(n-\tr{\Theta^2}\right) \right] \right],
\end{eqnarray*}
for $\nu = 1, \ldots, n$.

(i) By definitions of $\varphi_\nu$, $\chi_{\nu}$ and $\gamma_\nu$, we
obtain,
\begin{eqnarray*}
    \gamma_\nu &=& \arcsin(\varphi_\nu), \quad \gamma_\nu +
    \Delta\gamma_\nu = \arcsin(\varphi_\nu + \Delta\varphi_\nu) \quad and
    \quad\\
    \gamma_\nu &=& \arccos(\chi_\nu), \quad \gamma_\nu + \Delta\gamma_\nu
    = \arccos(\chi_\nu + \Delta\chi_\nu).
\end{eqnarray*}
By trigonometric identities, the difference between the above equations
admit,
\begin{eqnarray*}
    \sin(\Delta\gamma_\nu) &=& \sin(\arcsin(\varphi_\nu +
    \Delta\varphi_\nu) - \arcsin(\varphi_\nu)) \\
    &=& \sin\arcsin(\varphi_\nu + \Delta\varphi_\nu)
    \cos\arcsin(\varphi_\nu) - \cos\arcsin(\varphi_\nu
    + \Delta\varphi_\nu) \sin\arcsin(\varphi_\nu) \\
    &=& (\varphi_\nu + \Delta\varphi_\nu) \chi_\nu - (\chi_\nu +
    \Delta\chi_\nu) \varphi_\nu \\
    &=& \Delta\varphi_\nu \chi_\nu - \Delta\chi_\nu \varphi_\nu.
\end{eqnarray*}

Adopting the inequalities in \cref{cor:error-bounds} and the equality
recursively, we obtain
\begin{eqnarray*}
    \sin\left|\vartheta_\nu - \tvartheta_\nu\right| =
    \left|\sin(\Delta\gamma_\nu)\right| \leq |\Delta\varphi_\nu| +
    |\Delta\chi_\nu| \leq  2\calE,
\end{eqnarray*}
and, hence,
\begin{eqnarray*}
    \left| \vartheta_\nu - \tvartheta_\nu \right| \leq  \arcsin(2\calE).
\end{eqnarray*}

(ii) By the Taylor expansion of $P_{1,\nu}$ at $(\varphi_1, \varphi_2,
\ldots, \varphi_n)$, we obtain
\begin{eqnarray*}
    \left|\tP_{1,\nu} - P_{1,\nu}\right| & = &
    \left|\sum_{\nu=1}^n \Delta \varphi_\nu\frac{\partial
    P_{1,\nu}}{\partial\varphi_\nu} + o(\Delta\varphi_\nu)\right|
    \leq
    |\Delta\varphi|\frac{2\varphi_\nu\sum_{k \neq \nu} \varphi_k^2}
    {\left(\sum_{k=1}^n \varphi_k^2\right)^2} + o(\Delta\varphi)
    \leq \frac{2\calE\varphi_\nu}
    {\sum_{k=1}^n \varphi_k^2} + o(\calE),
\end{eqnarray*}
where the second inequality adopts \cref{cor:error-bounds}. The bound for
$\left|\tP_{2,\nu} - P_{2,\nu}\right|$ could be derived similarly.

(iii) By the Taylor expansion of $D_1$ at $(\varphi_1, \varphi_2, \ldots,
\varphi_n)$,
\begin{eqnarray*}
    \tilde{D}_1  = D_1 +  \sum_{i=1}^n \Delta \varphi_i \frac{\partial
    D_1}{\partial \varphi_i} + o(\Delta \varphi_i),
\end{eqnarray*}
where
\begin{eqnarray*}
    \frac{\partial D_1}{\partial \varphi_{i}} &=& -\frac{1}{\log{n}}
    \left[\frac{2\varphi_i}{\sum_{k=1}^n\varphi_k^2}
    \left(\log\frac{\varphi_i^2}{\sum_{k=1}^n \varphi_k^2} + 1\right)
    -\sum_{j=1}^n \frac{2\varphi_i \varphi_j^2} {(\sum_{k=1}^n
    \varphi_k^2)^2}
    \left(\log\frac{\varphi_j^2}{\sum_{k=1}^n \varphi_k^2} + 1\right)
    \right]\\
    &=& -\frac{1}{\log{n}} \left[\frac{2\varphi_i}{\sum_{k=1}^n
    \varphi_k^2} \left(\log\frac{\varphi_i^2}{\sum_{k=1}^n \varphi_k^2} +
    1\right) -\frac{2\varphi_i}{\sum_{k=1}^n \varphi_k^2}\left(D_1 \log{n}
    + 1\right) \right]\\
    &=& -\frac{1}{\log{n}} \left[\frac{2\varphi_i}{\sum_{k=1}^n
    \varphi_k^2} \left(\log\frac{\varphi_i^2}{\sum_{k=1}^n \varphi_k^2} +
    D_1 \log{n}\right) \right].
\end{eqnarray*}
We obtain
\begin{eqnarray*}
    \left|D_1 - \tD_1\right| & = & \left| \sum_{i=1}^n \Delta \varphi_i
    \frac{\partial D_1}{\partial \varphi_i} + o(\Delta \varphi_i)\right|
    \leq \left|\Delta \varphi\right| \left| \sum_{i=1}^n \frac{\partial
    D_1}{\partial \varphi_i}\right| + o(\Delta \varphi)\\
    & \leq & 2\calE \sum_{i=1}^n \left|\frac{\varphi_i}{\sum_{k=1}^n
    \varphi_k^2} \left(\log\frac{\varphi_i^2}{\sum_{k=1}^n
    \varphi_k^2}\frac{1}{\log{n}} + D_1\right) \right| + o(\calE),\\
\end{eqnarray*}
where second inequality adopts \cref{cor:error-bounds}.

The bound for $\left|\tD_2 - D_2\right|$ could be derived similarly.
\end{proof}

\section{Numerical experiments}
\label{sec:numerical-results}

We apply \cref{alg:rand-gsv} in comparative analysis of both synthetic
data sets and genome-scale expression data sets from practice. Comparative
analysis quantities, $\vartheta_\nu, P_{i,\nu}, D_i$ for $i = 1, 2$ and $1
\leq \nu \leq n$, are evaluated following a GSV calculation. All numerical
experiments are carried out on MATLAB R2021b with machine epsilon being
around $2.2204 \times 10^{-16}$. By default, we adopt MATLAB {\tt gsvd}
results as referneces.
The source code of our method is released at \url{https://github.com/shenwj87/RGSVsA.git}.

\subsection{Synthetic data sets}
\label{subsec4.1}

Synthetic data sets are adopted to demonstrate the efficiency of
\cref{alg:rand-gsv}. We compare \cref{alg:rand-gsv} with the algorithm in
\cite{f05}, Riemann Newton (RN) method~\cite{xnb20}, the MATLAB built-in
functions {\tt gsvd} and {\tt economy-sized gsvd}.

The synthetic data sets are generated as follows. Here we give the rank of
$G_{1}$ and the rank of $G_{2}$ both being 60\% of $\min\{m,p,n\}$. The
generalized singular values that are neither one nor zero among
$\alpha_1^\star, \ldots, \alpha_n^\star$ are sampled from a random uniform
distribution after sorting. Then $\beta_i^\star$ is calculated such that
$(\alpha_i^\star)^2 + (\beta_i^\star)^2 = 1$ for $1 \le i \le n$. The
nonsingular matrix $R_\star \in \bbC^{n \times n}$ is an $n$-by-$n$ matrix
of normally distributed random complex numbers. The unitary matrices
$U_\star \in \bbC^{m \times n}$, $V_\star \in \bbC^{p\times n}$ are
orthonormalized Gaussian random complex matrices. The data set $\{G_1,
G_2\}$ is then $G_1 = U_\star \Sigma_1^\star R_\star$ and $G_2 = V_\star
\Sigma_2^\star R_\star$, where $\Sigma_1^\star = \diag{\alpha_1^\star,
\ldots, \alpha_n^\star}$ and $\Sigma_2^\star = \diag{\beta_1^\star,
\ldots, \beta_n^\star}$. Various $(m, p, n)$ choices are explored.
Absolute errors of GSVs are used to compare the accuracy of GSV
algorithms. Numerical results are reported in \cref{fig:syntheticresult},
\cref{tab:syntheticresult}, and \cref{fig:normfig}.

\begin{table}[htp]
    \centering
    \caption{Runtime (second) and absolute errors for synthetic data
    sets. The shortest runtimes are bolded.} \label{tab:syntheticresult}
    \begin{tabular}{ccrcc}
    \toprule
    $(m,p,n)$ & Algorithm & Runtime & $\fnorm{\Sigma_1^\star - \Sigma_1}$
    &$\fnorm{\Sigma_2^\star - \Sigma_2}$ \\
    \toprule
    \multirow{5}{*}{(10000,10000,10000)}
    & \cref{alg:rand-gsv}      & \textbf{105.48} & 4.45E$-$12 & 3.10E$-$11 \\
    & {\tt economy-sized gsvd} &         377.82  & 3.88E$-$12 & 3.13E$-$12 \\
    & {\tt gsvd}               &         365.21  & 3.88E$-$12 & 3.13E$-$13 \\
    & Algorithm in \cite{f05}  &        1220.28  & 1.36E$-$09 & 4.12E$-$09 \\
    & RN method~\cite{xnb20}   &         375.81  & 1.02E$-$12 & 2.35E$-$12 \\
    \midrule
    \multirow{5}{*}{(8010,4005,4000)}
    & \cref{alg:rand-gsv}      & \textbf{16.06} & 4.45E$-$14 & 1.27E$-$12 \\
    & {\tt economy-sized gsvd} &         22.83  & 9.19E$-$12 & 7.82E$-$12 \\
    & {\tt gsvd}               &         49.50  & 9.51E$-$12 & 6.44E$-$12 \\
    & Algorithm in \cite{f05}  &         63.82  & 6.84E$-$09 & 5.66E$-$09 \\
    & RN method~\cite{xnb20}   &         36.27  & 1.59E$-$12 & 1.00E$-$12 \\
    \midrule
    \multirow{5}{*}{(9010,9005,5000)}
    & \cref{alg:rand-gsv}      & \textbf{36.90} & 6.82E$-$14 & 7.20E$-$14 \\
    & {\tt economy-sized gsvd} &         44.97  & 8.22E$-$12 & 7.93E$-$12 \\
    & {\tt gsvd}               &         61.42  & 8.50E$-$12 & 7.25E$-$12 \\
    & Algorithm in \cite{f05}  &        158.44  & 7.19E$-$09 & 3.62E$-$09 \\
    & RN method~\cite{xnb20}   &         57.71  & 1.60E$-$12 & 1.36E$-$12 \\
    \midrule
    \multirow{5}{*}{(8000,8010,8005)}
    & \cref{alg:rand-gsv}      & \textbf{74.29} & 3.65E$-$10 & 6.36E$-$10 \\
    & {\tt economy-sized gsvd} &        150.78  & 5.06E$-$12 & 4.65E$-$12 \\
    & {\tt gsvd}               &        189.10  & 4.79E$-$12 & 3.92E$-$12 \\
    & Algorithm in \cite{f05}  &        647.37  & 1.08E$-$05 & 1.04E$-$05 \\
    & RN method~\cite{xnb20}   &        165.87  & 1.46E$-$12 & 1.72E$-$12 \\
    \midrule
    \multirow{5}{*}{(10000,5010,5000)}
    & \cref{alg:rand-gsv}      & \textbf{6.45} & 1.78E$-$15 & 1.65E$-$15 \\
    & {\tt economy-sized gsvd} &        41.21  & 2.69E$-$12 & 2.21E$-$12 \\
    & {\tt gsvd}               &        48.93  & 2.84E$-$12 & 1.78E$-$12 \\
    & Algorithm in \cite{f05}  &       175.65  & 8.76E$-$06 & 1.82E$-$05 \\
    & RN method~\cite{xnb20}   &        44.79  & 2.72E$-$12 & 2.46E$-$12 \\
    \midrule
    \multirow{5}{*}{(10010,5000,10000)}
    & \cref{alg:rand-gsv}      & \textbf{8.87} & 2.20E$-$15 & 2.04E$-$15 \\
    & {\tt economy-sized gsvd} &       197.49  & 2.36E$-$12 & 2.90E$-$12 \\
    & {\tt gsvd}               &       175.78  & 2.89E$-$12 & 2.44E$-$12 \\
    & Algorithm in \cite{f05}  &       298.85  & 5.20E$-$08 & 4.28E$-$08 \\
    & RN method~\cite{xnb20}   &       182.62  & 2.67E$-$12 & 2.85E$-$12 \\
    \midrule
    \multirow{5}{*}{(7000,7010,10000)}
    & \cref{alg:rand-gsv}      & \textbf{56.21} & 3.39E$-$14 & 3.77E$-$14 \\
    & {\tt economy-sized gsvd} &        117.85  & 2.78E$-$12 & 1.51E$-$12 \\
    & {\tt gsvd}               &        118.75  & 2.80E$-$12 & 2.57E$-$12 \\
    & Algorithm in \cite{f05}  &        429.30  & 3.76E$-$07 & 3.71E$-$07 \\
    & RN method~\cite{xnb20}   &        118.06  & 2.93E$-$12 & 2.86E$-$12 \\
    \midrule
    \multirow{5}{*}{(5000,10000,11000)}
    & \cref{alg:rand-gsv}      & \textbf{133.75} & 1.39E$-$14 & 1.46E$-$15 \\
    & {\tt economy-sized gsvd} &         195.23  & 1.17E$-$12 & 2.40E$-$12 \\
    & {\tt gsvd}               &         192.18  & 1.09E$-$12 & 3.72E$-$12 \\
    & Algorithm in \cite{f05}  &        1110.65  & 4.40E$-$07 & 4.31E$-$07 \\
    & RN method~\cite{xnb20}   &         194.31  & 6.11E$-$12 & 8.14E$-$12 \\
    \bottomrule
    \end{tabular}
\end{table}

\begin{figure}[htp]
    \centering
    \includegraphics[width=2.9in]{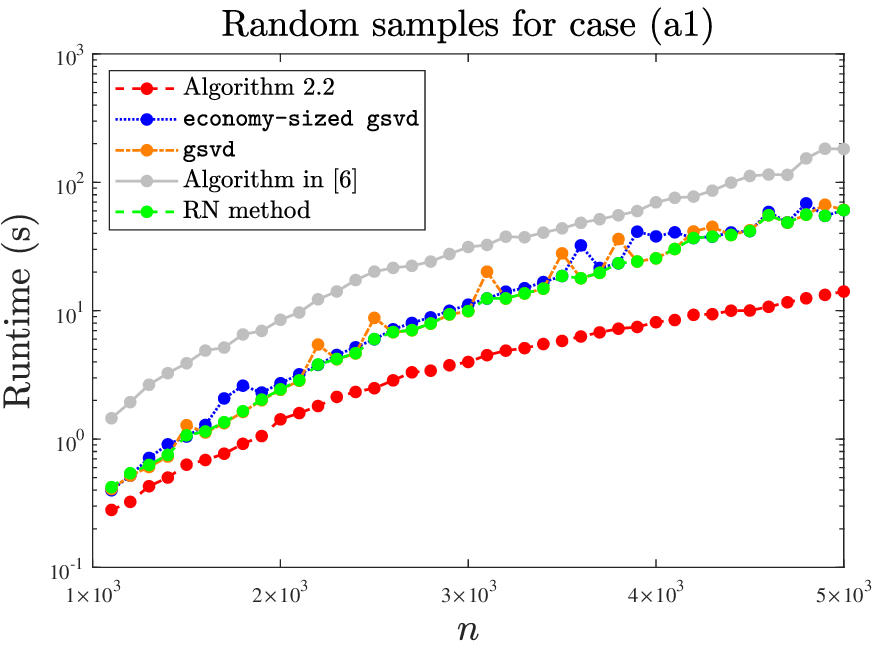}\;\;\;\;
    \includegraphics[width=2.9in]{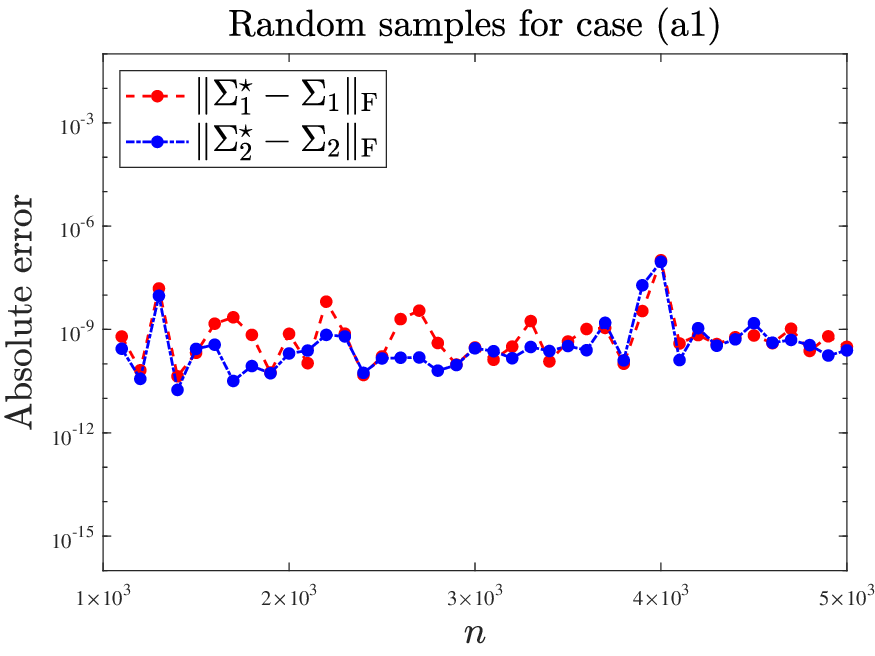}\\
    \includegraphics[width=2.9in]{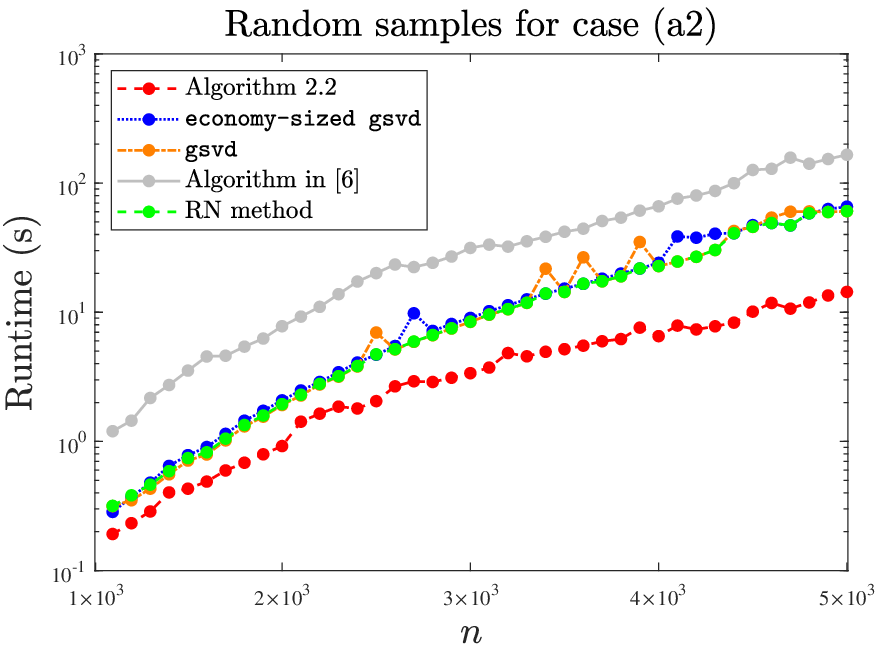}\;\;\;\;
    \includegraphics[width=2.9in]{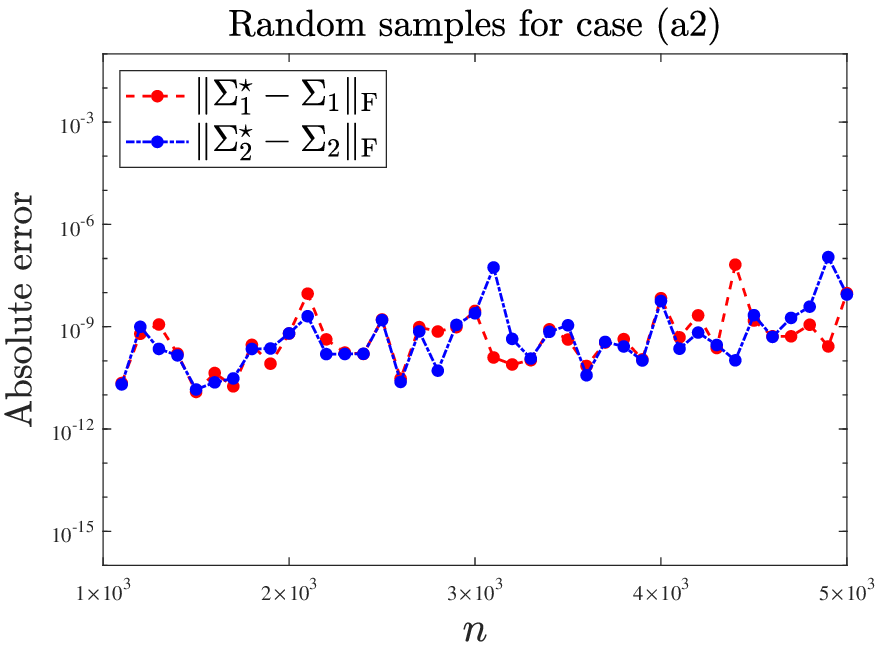}\\
    \includegraphics[width=2.9in]{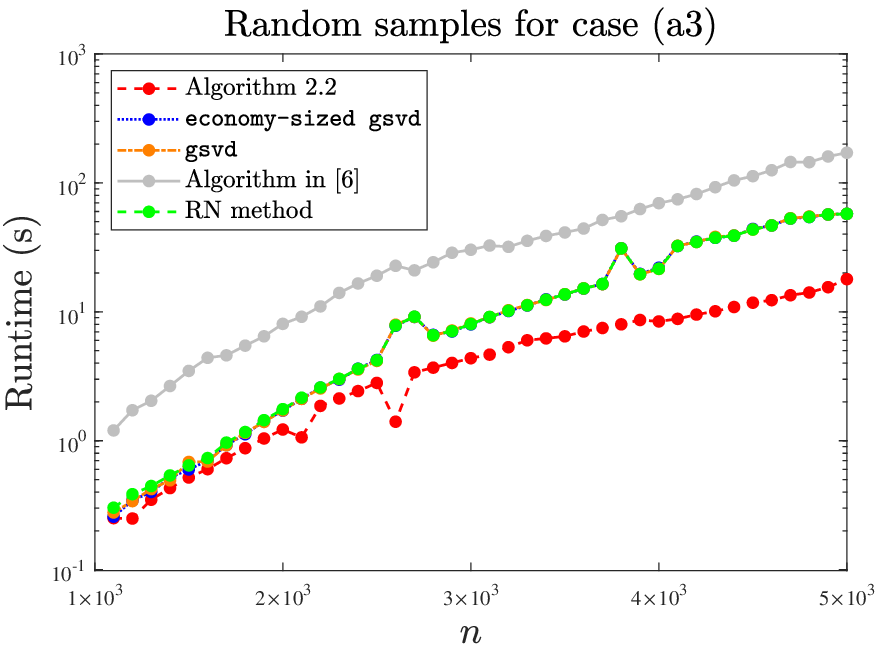}\;\;\;\;
    \includegraphics[width=2.9in]{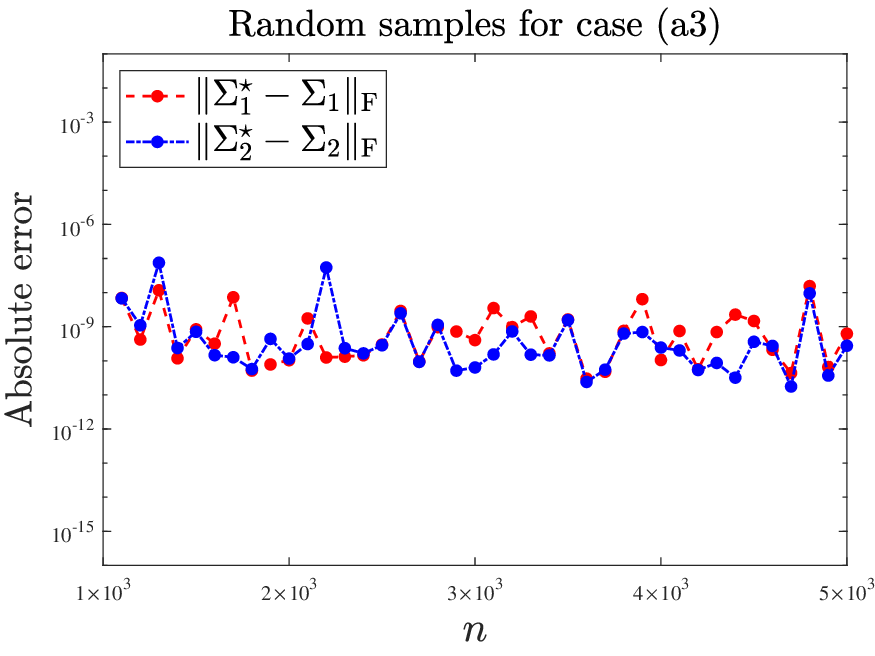}\\
    \caption{Runtime (second) and absolute errors for cases of $(m, p,
    n)$. (a1) with $m = n + 100$ and $p = n + 5$; (a2) with $m = n + 100$
    and $p = n - 5$; (a3) with $m = n - 100$ and $p = n - 5$.}
    \label{fig:syntheticresult}
\end{figure}

\begin{figure}[htb]
    \centering
    \includegraphics[width=3in]{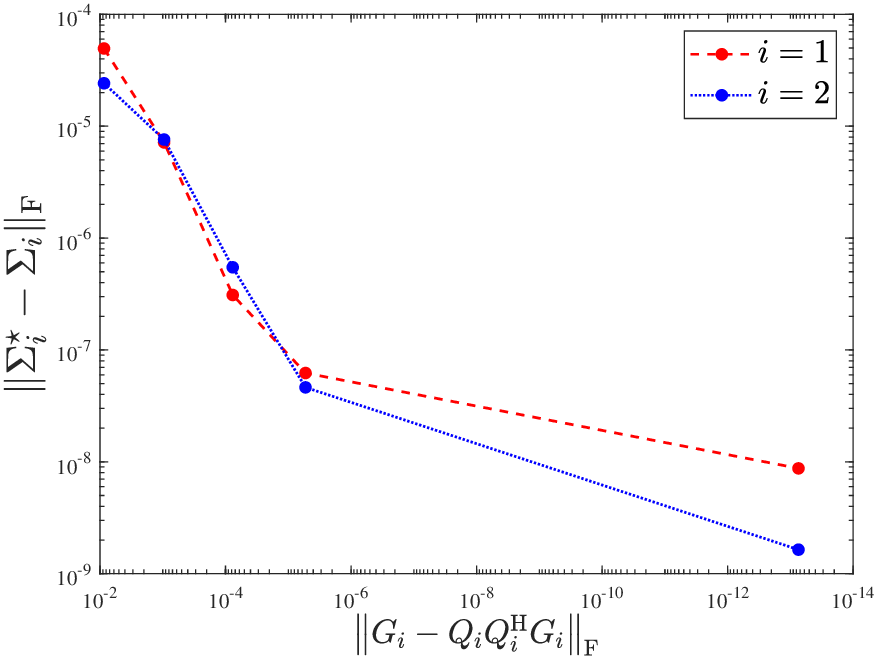}
    \caption{Relation between $\fnorm{G_i - Q_i Q_i^{\H} G_i}$ and
    $\fnorm{\Sigma_i^\star - \Sigma_i}$, for $i = 1, 2$.}
    \label{fig:normfig}
\end{figure}

In \cref{tab:syntheticresult}, we could observe the advantage of
\cref{alg:rand-gsv} both in runtime and accuracy. The runtimes of
\cref{alg:rand-gsv} are the shortest in all cases we have tested. In some
cases, it is 10x to 20x faster than the second-fastest algorithm. In the
least case, \cref{alg:rand-gsv} saves about 20\% runtime. Regarding
accuracy, \cref{alg:rand-gsv} achieves the best accuracy in most of the
cases. In the worst case, \cref{alg:rand-gsv} achieves $10^{-10}$ absolute
accuracy comparing to $10^{-12}$ of the best. Such an accuracy is
sufficient in almost all applications.

In \cref{fig:syntheticresult}, we explore the performance of various GSV
algorithms on matrices with increasing $n$. According to three figures in
the right column of \cref{fig:syntheticresult}, \cref{alg:rand-gsv}
achieves sufficiently high accuracy for comparative analysis problems.
\Cref{alg:rand-gsv}, as in the left column of \cref{fig:syntheticresult},
is the fastest among five algorithms. As the matrix size increases, the
runtime gap between \cref{alg:rand-gsv} and other algorithms further
enlarges.

As shown in \cref{fig:normfig}, when the basis approximation error
$\fnorm{G_i - Q_i Q_i^{\H} G_i}$ decreases, the absolute errors of GSVs
$\fnorm{\Sigma_i^\star - \Sigma_i}$ decrease for $i = 1, 2$. If only a few
digits of accuracy is needed for GSVs, which is the usual case in
practice, we could adopt a small number of bases in the approximation,
i.e., $Q_i$ with a small number of columns. The computational cost could
then be further reduced.

\subsection{Genome-scale expression data sets}

\Cref{alg:rand-gsv} is applied to two practical genome-scale expression
data sets in this section: yeast and human cell-cycle expression data set
and mice macrophage gene expression data set.

\subsubsection{Yeast and human cell-cycle expression data set}
\label{sec:yeast-human}

A yeast and human cell-cycle expression data set is adopted, which is
available at \url{http://genome-www.stanford.edu/GSVD/}. In this data set,
4523-genes $\times$ 18-arrays are analyzed for yeast and 12056-genes
$\times$ 18-arrays are analyzed for human. Hence, the matrix $G_1$ and
$G_2$ are of size $4523 \times 18$ and $12056 \times 18$, respectively.
Numerically, we validate that matrix $G_1$, $G_2$, and $\begin{pmatrix}
G_1 \\ G_2 \end{pmatrix}$ are of full column rank. Notice that some data
in the data set are missing. We adopt two methods~\cite{abb03}, the SVD
interpolation and spline, to recover these data. The runtimes of various
algorithms are reported in \cref{tab:yeast-human-runtime} and the
accuracies of \cref{alg:rand-gsv} for various comparative analysis
quantities are given in \cref{tab:yeast-human-error}.
\Cref{fig:yeast-human-svd-interpolation} and \cref{fig:yeast-human-spline}
illustrate the comparative analysis quantities of the SVD interpolation
and spline methods, respectively.

\begin{table}[htb]
    \center
    \caption{Runtime (second) of algorithms on the yeast and human
    cell-cycle expression data set with SVD interpolation and spline.}
    \label{tab:yeast-human-runtime}
    \begin{tabular}{lrr}
    \toprule
    Algorithm & SVD interpolation & Spline \\
    \toprule
    \cref{alg:rand-gsv}      & 0.000481 & 0.000470 \\
    {\tt economy-sized gsvd} & 0.003390 & 0.003914 \\
    {\tt gsvd}               & 3.808657 & 3.607354 \\
    Algorithm in \cite{f05}  & 0.002306 & 0.002280 \\
    RN method~\cite{xnb20}   & 0.003321 & 0.003868 \\
    \bottomrule
    \end{tabular}
\end{table}

\begin{table}[ht]
    \center
    \caption{Absolute errors of \cref{alg:rand-gsv} on the yeast and human
    cell-cycle expression data set with SVD interpolation and spline.}
    \label{tab:yeast-human-error}
    \begin{tabular}{lccccc}
    \toprule
    & \multicolumn{5}{c}{Absolute error} \\
    \cmidrule{2-6}
    & $\vartheta_\nu$ & $D_1$ & $D_2$ & $P_{1, \nu}$ & $P_{2, \nu}$ \\
    \toprule
    SVD interpolation & 4.13E$-$14 & 1.44E$-$15 & 1.33E$-$15 & 3.55E$-$15
    & 3.06E$-$15 \\
    Spline            & 1.99E$-$14 & 3.33E$-$16 & 6.66E$-$16 & 2.47E$-$15
    & 1.53E$-$15 \\
    \bottomrule
    \end{tabular}
\end{table}

\begin{figure}[htb]
    \centering
    \includegraphics[width=5.5in]{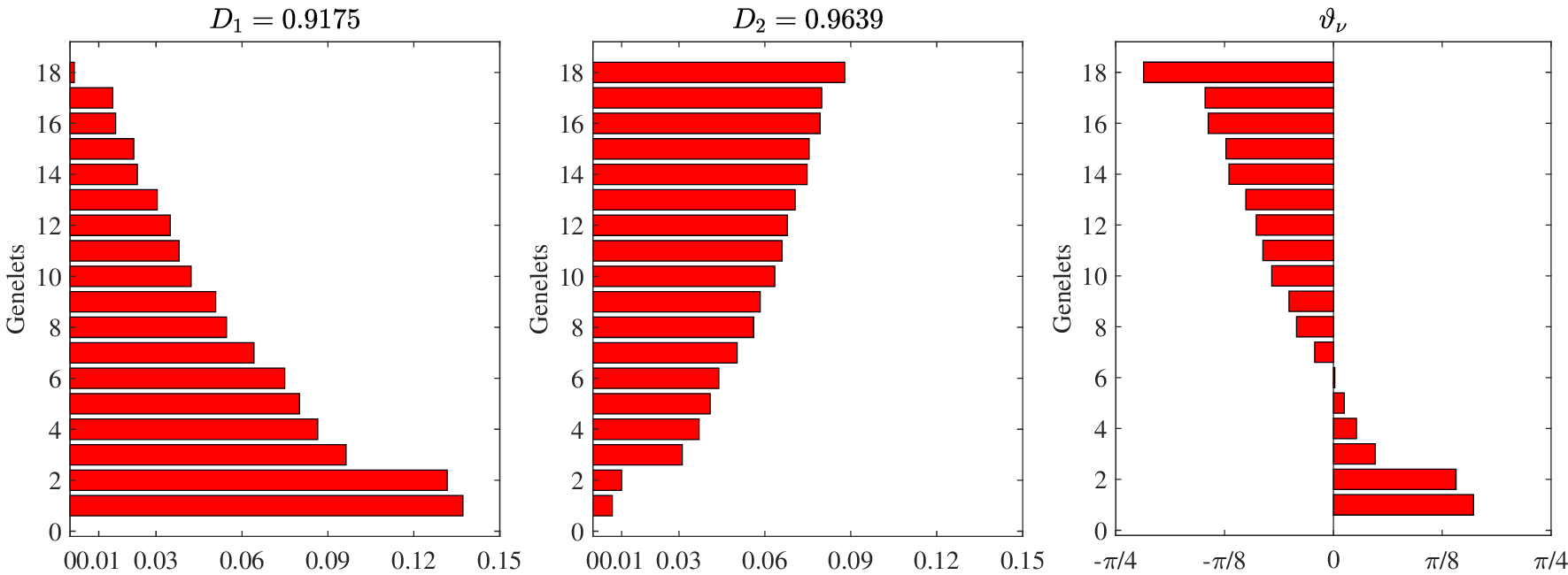}
    \caption{$P_{i,\nu}$, $D_i$ and $\vartheta_{\nu}$ computed by
    \cref{alg:rand-gsv} for yeast and human cell-cycle expression data set
    with SVD interpolation.}
    \label{fig:yeast-human-svd-interpolation}
\end{figure}

\begin{figure}[htb]
    \centering
    \includegraphics[width=5.5in]{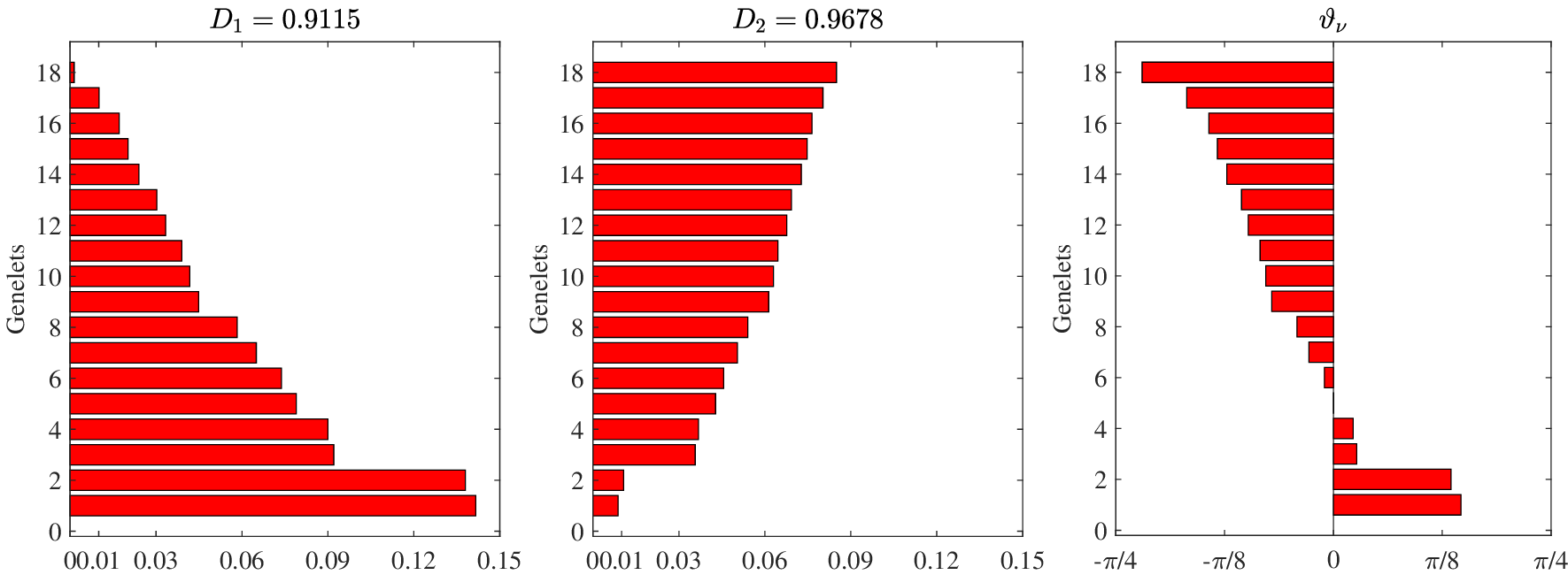}
    \caption{$P_{i,\nu}$, $D_i$ and $\vartheta_{\nu}$ computed by
    \cref{alg:rand-gsv} for yeast and human cell-cycle expression data set
    with spline.}
    \label{fig:yeast-human-spline}
\end{figure}

\Cref{tab:yeast-human-runtime} shows that the proposed
algorithm~\cref{alg:rand-gsv} is about five times faster than the second
fastest algorithms. The absolute accuracies obtained by
\cref{alg:rand-gsv}, as shown in \cref{tab:yeast-human-error}, are all
around machine epsilon, which is accurate enough in practical comparative
analysis. In this data set, the number of columns, i.e., $n$, is much
smaller than $m$ and $p$. Hence, both $G_1$ and $G_2$ matrices are
tall-and-skinny. All explored algorithms except {\tt gsvd} benefit from
the tall-and-skinny property in the data set and their runtimes are
significantly smaller than that of {\tt gsvd}.

According to \cref{fig:yeast-human-svd-interpolation} and
\cref{fig:yeast-human-spline}, yeast generalized fractions of
eigenexpression shows that the first two arrays capture more than $12\%$
of the overall yeast expression and human generalized fractions of
eigenexpression shows that the last array captures about $9\%$. When different
missing data recovery methods are adopted, the comparative analysis
results differ slightly. When the SVD interpolation is adopted, the sixth
array is equally significant in both data sets with $\vartheta_6 \approx
0$. When the spline is adopted, the $\vartheta$ of the fifth array is the
most close to zero.

\subsubsection{Mice macrophage gene expression data set}

After the mice macrophage with Polyamide (PA) and RPMI1640 medium
(containing phenol red) experiment, we can obtain data sets of the gene
mRNA expression level. This data set for mice macrophage with Polyamide
(PA) stimulation tabulates the matrix of size 22580-genes$\times$
9000-arrays and with RPMI1640 medium tabulates the matrix of size
22580-genes$\times$ 9000-arrays. Compared to the data set in
\cref{sec:yeast-human}, the data set in this section has much comparable
$m$, $p$, and $n$. data. The runtimes of various algorithms are reported
in \cref{tab:mice-runtime} and the accuracies of \cref{alg:rand-gsv} for
various comparative analysis quantities are given in
\cref{tab:mice-error}. \Cref{fig:mice} illustrates the histogram of
comparative analysis quantities.

\begin{table}[htb]
    \center
    \caption{Runtime (second) of algorithms on the mice macrophage gene
    expression data set.}
    \label{tab:mice-runtime}
    \begin{tabular}{lr}
    \toprule
    Algorithm & Runtime\\
    \toprule
    \cref{alg:rand-gsv}      &   400.06 \\
    {\tt economy-sized gsvd} & 40211.43 \\
    {\tt gsvd}               & 92874.80 \\
    Algorithm in \cite{f05}  & 38783.35 \\
    RN method~\cite{xnb20}   & 38932.93 \\
    \bottomrule
    \end{tabular}
\end{table}

\begin{table}[htb]
    \center
    \caption{Absolute errors of \cref{alg:rand-gsv} on the mice macrophage
    gene expression data set.}
    \label{tab:mice-error}
    \begin{tabular}{ccccc}
    \toprule
    \multicolumn{5}{c}{Absolute error} \\
    \midrule
    $\vartheta_\nu$ & $D_1$ & $D_2$ & $P_{1, \nu}$ & $P_{2, \nu}$ \\
    \toprule
    4.90E$-$12 & 7.21E$-$12 & 3.09E$-$12 & 6.90E$-$12 & 7.03E$-$12 \\
    \bottomrule
    \end{tabular}
\end{table}

\begin{figure}[htb]
    \centering
    \includegraphics[width=5.5in]{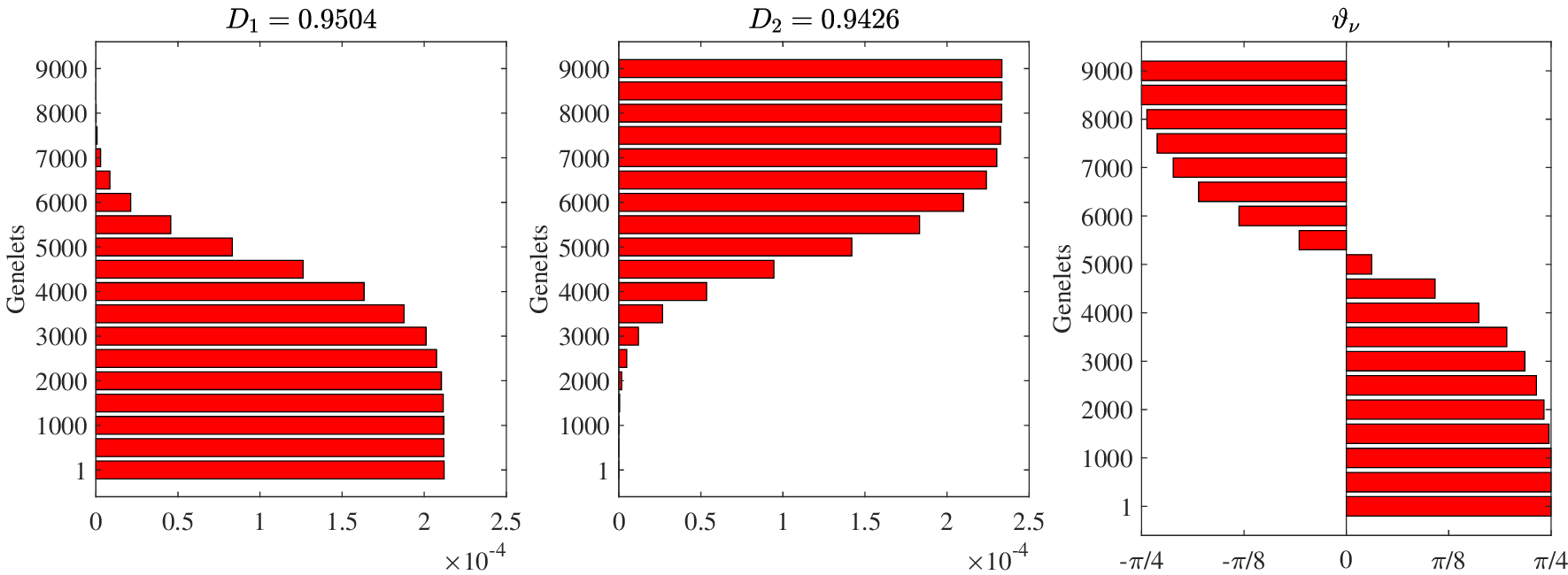}
    \caption{$P_{i,\nu}$, $D_i$ and $\vartheta_{\nu}$ computed by
    \cref{alg:rand-gsv} for mice macrophage gene expression data set.}
    \label{fig:mice}
\end{figure}

\Cref{tab:mice-runtime} shows that the proposed
algorithm~\cref{alg:rand-gsv} is at least one hundred times faster than
all other algorithms. The runtimes of other algorithms in this data set
are beyond 10 hours. While \cref{alg:rand-gsv} could obtain desired
results in less than 10 mins. The absolute accuracies obtained by
\cref{alg:rand-gsv}, as shown in \cref{tab:mice-error}, are all around
$10^{-12}$.

In this data set, the first 2000 genelets are highly significant in PA
gene expression relative to the RPMI1640 medium gene expression. The
4500th to 5000th genelets are almost equally significant in both, with a
slightly higher significance in the PA gene expression. The 8500th to
9000th genelets are highly significant in the RPMI1640 medium gene
expression data. The randomized projection in \cref{alg:rand-gsv} mostly
benefits from genelets that are significant in one gene expression, where
either $\alpha_i$ or $\beta_i$ is compressible. In many comparative
analysis data set, many genelets have biased significantly towards one
side. Hence, we conclude \cref{alg:rand-gsv} is an efficient algorithm for
calculating GSVs of comparative analysis data sets.

\section{Conclusion}
\label{sec:conclusion}

The target of this paper is to efficiently address the GSV problems in
comparative analysis. A randomized GSV algorithm is proposed, where the
key is to approximate bases of both $G_1$ and $G_2$ matrices by a randomized basis extraction
algorithm. By the overall procedure algorithm, generalized fractions of
eigenexpression and generalized normalized Shannon entropy for comparative
analysis of a class of genome-scale expression data sets can be
efficiently computed. The approximation accuracy of the randomized basis
extraction algorithm is analyzed. Combined with sensitivity analysis of
GSVs, we prove the error analysis of various comparative analysis
quantities. Finally, for both synthetic data sets and practical
genome-scale expression data sets, we demonstrate that our algorithm
outperforms other existing GSV algorithms in runtime. And the accuracy of
our algorithm is sufficient for the comparative analysis tasks.


\begin{thebibliography}{10}

\bibitem{abb00}
{\sc O.~Alter, P.~O. Brown, and D.~Botstein}, {\em Singular value decomposition
  for genome-wide expression data processing and modelling}, Proc. Nat. Acad.
  Sci., 97 (2000), pp.~10101--10106.

\bibitem{abb03}
{\sc O.~Alter, P.~O. Brown, and D.~Botstein}, {\em Generalized singular
  decomposition for comparative analysis of genome-scale expression data sets
  of two different organisms}, Proc. Nat. Acad. Sci., 100 (2003),
  pp.~3351--3356.

\bibitem{5}
{\sc Z.~Bai and J.~W. Demmel}, {\em Computing the generalized singular value
  decomposition}, SIAM J. Sci. Comput., 14 (1993), pp.~1464--1486.

\bibitem{d98}
{\sc Z.~Drma\v{c}}, {\em A tangent algorithm for computing the generalized
  singular value decomposition}, SIAM J. Numer. Anal., 35 (1998),
  pp.~1804--1832.

\bibitem{el89}
{\sc L.~M. Ewerbring and F.~T. Luk}, {\em Canonical correlations and
  generalized svd: applications and new algorithms}, Comput. Appl. Math., 27
  (1989), pp.~37--52.

\bibitem{f05}
{\sc S.~Friedland}, {\em A new approach to generalized singular value
  decomposition}, SIAM J. Matrix Anal. Appl., 27 (2005), pp.~434--444.

\bibitem{gv13}
{\sc G.~Golub and C.~V. Loan}, {\em Matrix Computations}, Johns Hopkins
  University Press, Baltimore, 4th~ed., 2013.

\bibitem{hmt11}
{\sc N.~Halko, P.~G. Martinsson, and J.~A. Tropp}, {\em Finding structure with
  randomness: probabilities algorithms for constructing approximate matrix
  decompositions}, SIAM Rev., 53 (2011), pp.~217--288.

\bibitem{hj91}
{\sc R.~A. Horn and C.~R. Johnson}, {\em Singular value inequalities},
  Cambridge University Press, 1991, pp.~134--238.

\bibitem{23}
{\sc C.~F.~V. Loan}, {\em Generalizing the singular value decomposition}, SIAM
  J. Numer. Anal., 13 (1976), pp.~76--83.

\bibitem{24}
{\sc C.~F.~V. Loan}, {\em Computing the cs and the generalized singular value
  decompositions}, Numer. Math., 46 (1985), pp.~479--491.

\bibitem{p21}
{\sc P.~Martinsson and J.~Tropp}, {\em Randomized numerical linear algebra:
  Foundations and algorithms}, arXiv:2002.01387,  (2020).

\bibitem{p84}
{\sc C.~C. Paige}, {\em A note on a result of sun ji-guang: Sensitivity of the
  cs and gsv decompositions}, SIAM Journal on Numerical Analysis, 21 (1984),
  pp.~186--191.

\bibitem{ps81}
{\sc C.~C. Paige and M.~A. Saunders}, {\em Towards a generalized singular value
  decomposition}, SIAM J. Numer. Anal., 18 (1981), pp.~398--405.

\bibitem{21}
{\sc G.~W. Stewart}, {\em Computing the cs-decomposition of a partitioned
  orthonormal matrix}, Numer. Math., 40 (1982), pp.~297--306.

\bibitem{xnb20}
{\sc W.~Xu, M.~Ng, and Z.~Bai}, {\em Geometric inexact newton method for
  generalized singular values of grassmann matrix pair}, SIAM J. Matrix Anal.
  Appl., 43 (2022), pp.~535--560.

\bibitem{z92}
{\sc H.~Zha}, {\em A numerical algorithm for computing restricted singular
  value decomposition of matrix triplets}, Linear Algebra Appl., 168 (1992),
  pp.~1--25.

\end{thebibliography}
\end{document}